\documentclass[12pt]{article}
\usepackage[T1]{fontenc}
\usepackage[utf8x]{inputenc}
\usepackage{amsmath, amssymb, amsthm}
\usepackage{mathrsfs}
\usepackage[table,x11names]{xcolor}
\usepackage[margin=1in]{geometry}
\usepackage{hyperref}
\usepackage{theoremref}
\usepackage{graphicx}
\usepackage{tikz}
\usepackage{array}
\usepackage{seqsplit}
\usepackage{fancyvrb}
\usepackage{todonotes}
\usepackage{float}

\usepackage{dirtytalk}

\newtheorem{prop}{Proposition}
\newtheorem{thm}{Theorem}
\newtheorem*{thm*}{Theorem}
\newtheorem{cor}{Corollary}

\newtheorem*{prob*}{Problem}
\theoremstyle{definition}
\newtheorem{defn}{Definition}
\newtheorem*{defn*}{Definition}
\theoremstyle{definition}
\newtheorem{exmp}{Example}
\newtheorem{conjecture}{Conjecture}

\usepackage{authblk}

\title{Every Binary Code Can Be Realized by Convex Sets}
\author[1]{Megan K. Franke}
\author[2]{Samuel Muthiah}

\affil[1]{Department of Mathematics, University of California Santa Barbara, Santa Barbara, CA 93106-3080}
\affil[2]{Department of Mathematics and Computer Science, Westmont College, Santa Barbara, CA 93108}

\date{April 14, 2018}

\begin{document}

\maketitle
\begin{abstract}
Much work has been done to identify which binary codes can be represented by collections of open convex or closed convex sets. While not all binary codes can be realized by such sets, here we prove that every binary code can be realized by convex sets when there is no restriction on whether the sets are all open or closed. We achieve this by constructing a convex realization for an arbitrary code with $k$ nonempty codewords in $\mathbb{R}^{k-1}$. This result justifies the usual restriction of the definition of convex neural codes to include only those that can be realized by receptive fields that are all either open convex or closed convex. We also show that the dimension of our construction cannot in general be lowered.
\end{abstract}

\section{Introduction}
In neuroscience, it is common to consider collections of coactive neurons as information-coding units in neural populations. Abstractly, we encode this data as a binary code.
\begin{defn}
A \textit{binary code} on $n$ neurons is a collection of subsets $\mathcal{C}$ of the set $[n] = \{1,2,3,\ldots , n\}$. The elements of $\mathcal{C}$ are called \textit{codewords}.
\end{defn}

The primary goal of this paper is to answer the question: can every binary code be realized by convex sets, not necessarily open or closed, in $\mathbb{R}^k$? Prior to addressing the motivation for this question, we provide an example of how a binary code is extracted from a collection of convex sets.

\begin{exmp}
Consider the following diagram where $U_1$, $U_2$, $U_3$, and $U_4$ are convex sets in a stimulus space $X$:
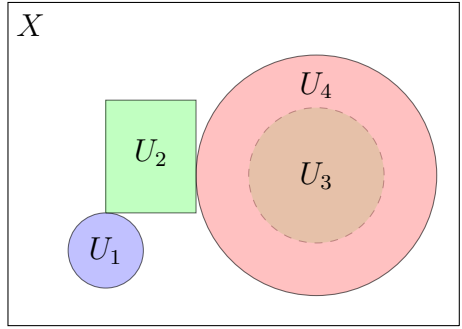
\begin{figure}[H]
\begin{center}
\begin{tikzpicture}
\draw [fill=blue!40, opacity=.6] (-1.7,-1) circle [radius = 0.5] {};
\draw [fill=green!40, opacity=.6, dashed] (1.1,0) circle [radius = .9] {};
\draw [fill=red!40, opacity=.6] (1.1,0) circle [radius = 1.6] {};
\draw [fill=green!40, opacity=.6] (-1.7, -0.5) -- (-0.5, -0.5) -- (-0.5, 1) -- (-1.7,1) -- (-1.7, -0.5);
\node at (1.1,0) {$U_3$};
\node at (-1.7,-1) {$U_1$};
\node at (1.1,1.2) {$U_4$};
\node at (-1.1,0.3) {$U_2$};
\node at (-2.7,2) {$X$};
\draw (3,-2) -- (-3,-2) -- (-3, 2.3) -- (3,2.3) -- (3,-2);
\end{tikzpicture}
\caption{Convex sets $U_1, U_2, U_3,$ and $U_4$, in the stimulus space $X$.}
\end{center}
\end{figure}
The corresponding code $\mathcal{C}$ of such a collection of sets is obtained by representing each unique region cut out by $U_1, U_2, U_3$, and $U_4$ as a binary string, where the $i$th entry is 1 if that region is in $U_i$ and 0 otherwise. For example, $(0,1,0,0) \in \mathcal{C}$ because there exists a region of $U_2$ such that $U_2$ is not intersecting any other $U_i$. Moreover, $(1,0,0,1) \not\in \mathcal{C}$ because $U_1$ and $U_4$ are disjoint. In this way, this diagram generates the following corresponding code: $\mathcal{C} =\{ (0,0,0,0), (1,0,0,0),$ $(0, 1, 0, 0), (0, 0, 0, 1), (1,1,0,0) , (0, 1, 0, 1), (0, 0, 1, 1)\}.$ We say that $\mathcal{C}$ is the binary code arising from our diagram. In this paper, we will write our code using the supports of the binary strings: $\mathcal{C}=\{\emptyset, 1, 2, 4, 12, 24, 34\}$. Since $U_1$, $U_2$, $U_3$, and $U_4$ are convex, we say that $\mathcal{C}$ is convex realizable. The focus of this paper will be on determining when a neural code is convex realizable.

\end{exmp}

	 Our study of binary codes is motivated by biological research into a type of neuron called a place cell. In 1971, John O'Keefe discovered place cells, an accomplishment for which he shared the 2014 Nobel Prize in Physiology or Medicine. A place cell is a neuron that fires when an animal is in a particular location relative to its environment and thus provides an internal representation of the animal's spatial location. These particular locations, called receptive fields, are approximately convex. From the regions cut out by the receptive fields, we obtain a binary code called a neural code \cite{what-makes}. 
     
     It has been observed experimentally that receptive fields arising from place cells are well approximated by sets that are not only convex but full dimensional. Thus, it is usually assumed that receptive fields are open sets. A neural code is generally said to be \say{convex} if there exist open convex sets which realize the code. Since we are focusing on convex sets in general and not open convex sets in particular, we will be using different terminology, to be defined soon.

Next, we repeat Example 1 under the assumption that the convex sets are receptive fields and thus open. 

\begin{exmp}
Consider the following diagram where $U_1$, $U_2$, $U_3$, and $U_4$ are open convex sets in a stimulus space $X$:
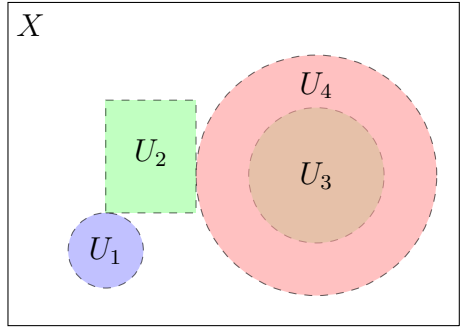
\begin{figure}[H]
\begin{center}
\begin{tikzpicture}
\draw [fill=blue!40, opacity=.6, dashed] (-1.7,-1) circle [radius = 0.5] {};
\draw [fill=green!40, opacity=.6, dashed] (1.1,0) circle [radius = .9] {};
\draw [fill=red!40, opacity=.6, dashed] (1.1,0) circle [radius = 1.6] {};
\draw [fill=green!40, opacity=.6, dashed] (-1.7, -0.5) -- (-0.5, -0.5) -- (-0.5, 1) -- (-1.7,1) -- (-1.7, -0.5);
\node at (1.1,0) {$U_3$};
\node at (-1.7,-1) {$U_1$};
\node at (1.1,1.2) {$U_4$};
\node at (-1.1,0.3) {$U_2$};
\node at (-2.7,2) {$X$};
\draw (3,-2) -- (-3,-2) -- (-3, 2.3) -- (3,2.3) -- (3,-2);
\end{tikzpicture}
\caption{Open convex sets $U_1, U_2, U_3,$ and $U_4$ in stimulus space $X$.}
\end{center}
\end{figure}

Note that by assuming openness, the intersections that have dimension less than 2 from Example 1 are no longer present. For instance, the codeword $(0,1,0,1)$ was present in the code of Example 1 but not in Example 2. From this diagram, we obtain the corresponding binary code: $ \mathcal{C}=\{ (0,0,0,0), (1,0,0,0),$ $(0, 1, 0, 0), (0, 0, 0, 1), (0, 0, 1, 1)\}$, which we write as $\{\emptyset, 1,2,4,34\}$. Since $U_1, U_2, U_3,$ and $U_4$ are open and convex, we say that the code $\mathcal{C}$ itself is open convex.  

\end{exmp}

We formalize these notions with the following definitions. This language allows us to phrase our biologically motivated questions about neural codes in geometric and combinatorial terms.

We introduce the following notation:
$$U_{\sigma} := \bigcap_{i\in\sigma} U_i.$$

\begin{defn}
A code $\mathcal{C}$ on $n$ neurons is \textit{convex realizable} if there exists a collection of convex sets, not necessarily open or closed, $\mathcal{U}=\{U_1,...,U_n\}$ in some stimulus space $X \subseteq \mathbb{R}^k$ such that
$\mathcal{C}=\mathcal{C}(\mathcal{U}):=\{\sigma \subseteq [n] \mid  U_{\sigma} \setminus \bigcup_{j\in [n]\setminus{\sigma}} U_j \neq \emptyset \}$ where $U_{\emptyset} := X$. The minimal $k$ such that $\mathcal{C}$ is convex realizable in $\mathbb{R}^k$ is called the \textit{minimal convex embedding dimension} of $\mathcal{C}$.
\end{defn}

\begin{defn}
If a code $\mathcal{C}$ is convex realizable by a set $\mathcal{U}$ where each $U_i \in \mathcal{U}$ is open, we say that $\mathcal{C}$ is \textit{open convex}. Similarly, if a code $\mathcal{C}$ is convex realizable by a set $\mathcal{U}$ where each $U_i \in \mathcal{U}$ is closed, we say that $\mathcal{C}$ is \textit{closed convex}.
\end{defn}

One of the primary goals of this research area is to determine which codes are open or closed convex realizable, as not all are \cite{what-makes,neural_ring,LSW}. Research focuses on open convex codes because, as mentioned earlier, full dimensional sets are the best representation of the receptive fields and openness guarantees this property. Cruz et al. showed that all max-intersection complete codes are both open and closed convex \cite{giusti-preprint}. Another area of significant interest has been the minimal dimension in which open connected codes have an open connected realization. Mulas and Tran proved that all open connected codes have an open connected realization in dimension at most 3 \cite{connected-codes}. Rosen and Zhang characterized codes that are realizable by open convex sets in dimension 1 \cite{d-1}. Less work has been done investigating convex codes without regard to openness or closedness because these codes are not directly motivated from the receptive fields of place cells. However, by exploring when a binary code is realizable by convex sets, not necessarily open or closed, we provide a frame of reference for the study of open convex or closed convex codes.

If a code is open convex or closed convex, then by definition it is convex realizable. But, are there many more convex realizable codes than open or closed convex realizable codes? Tancer showed that every simplicial complex is realized by convex sets \cite{tancer-survey}.  Some have speculated that, under a certain set of conditions, all codes are convex \cite{Chen}. Another reasonable conjecture is that all codes are convex realizable. Indeed, we will show that every binary code is convex realizable (Theorem 1). Following this result, we prove that the dimension of the convex realization provided in Theorem 1 cannot always be improved (Theorem 2). To finish, we conjecture a classification of codes which can be realized by convex sets in $\mathbb{R}^1$ (Conjecture 1).

\section{Main Results}

Our primary result is a construction of a convex realization in $\mathbb{R}^{k-1}$ of an arbitrary code $\mathcal{C}$, where $k$ is the number of nonempty codewords in $\mathcal{C}$. We will begin with an example of the construction, followed by a proof of the construction in Theorem 1. Following this result, the remainder of the paper explores the minimal convex embedding dimension of a code. In Theorem 2 we prove that, for a certain class of codes, the dimension of the construction in Theorem 1 is the minimal convex embedding dimension of the code. Finally, Conjecture 1 speculates a sufficient condition for the minimal convex embedding dimension of a convex open code to be 1 and provides a possible procedure to prove this result.

Throughout this work we use the notation conv$(A)$ to denote the convex hull of a subset $A\subseteq \mathbb{R}^k$.

\begin{exmp}
Consider the code $\mathcal{C} = \{\emptyset, 12, 34, 123\}$ on $n=4$ neurons. We will show how to realize this code using convex sets by using a method which, later proved in Theorem 1, produces a convex realization of any binary code.

\end{exmp}

\begin{figure}[htb]
\begin{center}
\vspace{.5cm}
\resizebox{4.83in}{!}{
\begin{tikzpicture}
\node (space) at (0,2.5) {};
\draw node[fill=black, inner sep=.1cm][label=below right:$V_{1}^1$] (v11) at (-7.5,2) [shape=circle]{};
\draw node[fill=black, inner sep=.1cm][label=below right:$V_{2}^1$] (v21) at (-3.5,2) [shape=circle]{};

\draw[-](8,-10)--(8,2.7)--(-10,2.7)-- (-10,-10)--(8,-10);
\draw[-] (-8, -10)--(-8,3.7)--(8,3.7)--(8,-10);
\draw[-] (-10, 1)--(8,1);
\draw[-] (-10,-1)--(8,-1);
\draw[-] (-10, -5.5)--(8,-5.5);

\draw[-](-4, -10)--(-4,3.7);
\draw[-](0, -10)--(0,3.7);
\draw[-](4, -10)--(4,3.7);

\node[label=below : Step 1] at (-9.1,2.3) {};
\node[label=below : Step 2] at (-9.1,.4) {};
\node[label=below : Step 3] at (-9.1,-2.7) {};
\node[label=below : Step 4] at (-9.1,-7.3) {};
\node[label=below : Neuron 1] at (-6, 3.5) {};
\node[label=below : Neuron 2] at (-2, 3.5) {};
\node[label=below : Neuron 3] at (2, 3.5) {};
\node[label=below : Neuron 4] at (6, 3.5) {};

\draw node[fill=none, draw=black, inner sep=.1cm] (ov32) at (.5, 0) [shape=circle] {};
\draw node[fill=black, draw=black, inner sep=.1cm] (cv32) at (3.5,0) [shape=circle] {};
\draw[-] (ov32)--(cv32);
\node[label=below : $V_{3}^2$] (v32) at (2,0) {};

\draw node[fill=none, draw=black, inner sep=.1cm] (ov42) at (4.5, 0) [shape=circle] {};
\draw node[fill=black, draw=black, inner sep=.1cm] (cv42) at (7.5,0) [shape=circle] {};
\draw[-] (ov42)--(cv42);
\node[label=below : $V_{4}^2$] (v42) at (6,0) {};

\fill[black!30] (-7.5,-4.5)--(-4.5,-4.5)--(-7.5,-1.5);
\draw node[fill=white, draw=black, inner sep=.1cm] (olv13) at (-7.5, -4.5) [shape=circle] {};
\draw node[fill=white, draw=black, inner sep=.1cm] (orv13) at (-4.5, -4.5) [shape=circle] {};
\draw[dashed] (olv13)--(orv13);
\draw[-] (olv13)--(-7.5,-1.5);
\draw[-] (-7.5,-1.5)--(orv13);
\node[label=below:$V_{1}^3$] (V1_3) at (-6,-4.5) {};
\draw node[fill=black, draw=black, inner sep=.1cm] (olv1) at (-7.5, -1.5) [shape=circle] {};

\fill[black!30] (-3.5,-4.5)--(-.5,-4.5)--(-3.5,-1.5);
\draw node[fill=white, draw=black, inner sep=.1cm] (olv23) at (-.5, -4.5) [shape=circle] {};
\draw node[fill=white, draw=black, inner sep=.1cm] (orv23) at (-3.5, -4.5) [shape=circle] {};
\draw[dashed] (olv23)--(orv23);
\draw[-] (olv23)--(-3.5,-1.5);
\draw[-] (-3.5,-1.5)--(orv23);
\node[label=below:$V_{2}^3$] (V2_3) at (-2,-4.5) {};
\draw node[fill=black, draw=black, inner sep=.1cm] (olv1) at (-3.5, -1.5) [shape=circle] {};

\fill[black!30] (.5,-4.5)--(3.5,-4.5)--(.5,-1.5);
\draw node[fill=white, draw=black, inner sep=.1cm] (olv33) at (.5, -4.5) [shape=circle] {};
\draw node[fill=white, draw=black, inner sep=.1cm] (orv33) at (3.5, -4.5) [shape=circle] {};
\draw[dashed] (olv33)--(orv33);
\draw[-] (olv33)--(.5,-1.5);
\draw[-] (.5,-1.5)--(orv33);
\node[label=below:$V_{3}^3$] (V33) at (2,-4.5) {};
\draw node[fill=black, draw=black, inner sep=.1cm] (olv1) at (.5, -1.5) [shape=circle] {};

\fill[black!30] (-7.5,-9)--(-4.5,-9)--(-7.5,-6);
\draw node[fill=white, draw=black, inner sep=.1cm] (crv1) at (-4.5, -9) [shape=circle] {};
\draw node[fill=black, draw=black, inner sep=.1cm] (olv1) at (-7.5, -9) [shape=circle] {};
\draw[dashed] (olv1)--(crv1);
\draw[-] (olv1)--(-7.5,-6);
\draw[-] (-7.5,-6)--(crv1);
\node[label=below:$U_{1}$] (V1) at (-6,-9) {};
\draw node[fill=black, draw=black, inner sep=.1cm] (olv1) at (-7.5, -6) [shape=circle] {};

\fill[black!30] (-3.5,-9)--(-.5,-9)--(-3.5,-6);
\draw node[fill=white, draw=black, inner sep=.1cm] (orv2) at (-.5, -9) [shape=circle] {};
\draw node[fill=black, draw=black, inner sep=.1cm] (clv2) at (-3.5, -9) [shape=circle] {};
\draw[dashed] (orv2)--(clv2);
\draw[-] (clv2)--(-3.5,-6);
\draw[-] (-3.5,-6)--(orv2);
\node[label=below:$U_{2}$] (V2) at (-2,-9) {};
\draw node[fill=black, draw=black, inner sep=.1cm] (olv1) at (-3.5, -6) [shape=circle] {};

\fill[black!30] (.5,-9)--(3.5,-9)--(.5,-6);
\draw node[fill=black, draw=black, inner sep=.1cm] (crv3) at (3.5, -9) [shape=circle] {};
\draw node[fill=white, draw=black, inner sep=.1cm] (olv3) at (0.5, -9) [shape=circle] {};
\draw[-] (crv3)--(olv3);
\draw[-] (olv3)--(.5,-6);
\draw[-] (.5,-6)--(crv3);
\node[label=below:$U_{3}$] (V3) at (2,-9) {};
\draw node[fill=black, draw=black, inner sep=.1cm] (olv1) at (.5, -6) [shape=circle] {};

\draw node[fill=white, draw=black, inner sep=.1cm] (olv4) at (4.5, -9) [shape=circle] {};
\draw node[fill=black, draw=black, inner sep=.1cm] (crv4) at (7.5,-9) [shape=circle] {};
\draw[-] (olv4)--(crv4);
\node[label=below : $U_{4}$] (v4) at (6,-9) {};

\end{tikzpicture}
}
\caption{Constructing a convex realization of the code $\mathcal{C}=\{\emptyset, 12, 34, 123\}$, as in the proof of Theorem 1.}
\end{center}
\end{figure}

The intuition behind this construction is that, if for each codeword we add a dimension to our realization, then convexity can always be preserved. First, ordering our nonempty codewords, let $\sigma_1 = 12$, $\sigma_2 = 34$, and $\sigma_3 = 123$. Let $\{e_1,e_2\}$ be the standard basis for $\mathbb{R}^2$. We complete the following steps, which we illustrate in Figure 3.\\
\begin{itemize}
\item Step 1: \emph{For each $j\in \sigma_1$, we define the set $V_j^1 =$conv$\{0\}$. Otherwise, we set $V_j^1 = \emptyset$.} Since $1,2 \in \sigma_1$, we take the sets $V_1^1$ and $V_2^1$ to be the point conv$\{0\}$. Moreover, since $3,4 \not\in \sigma_1$, we define $V_3^1$ and $V_4^1$ to be empty. 
\item Step 2: \emph{For each $j\in \sigma_2$, we define the set $V_j^2 =$conv$\{0, e_1\} \setminus$ conv$\{0\}$. Otherwise, we set $V_j^2 = \emptyset$.} Since $3,4 \in \sigma_2$, we define $V_3^2$ and $V_4^2$ to be the segment with open endpoint conv$\{0, e_1\} \setminus$ conv$\{0\}$. Moreover, since $1,2 \not\in \sigma_2$, we take the sets $V_1^2$ and $V_2^2$ to be empty.
\item Step 3: \emph{For each $j\in \sigma_3$, we define the set $V_j^3 =$conv$\{0, e_1, e_2\} \setminus$ conv$\{0, e_1\}$. Otherwise, we set $V_j^3 = \emptyset$.} Since $1,2,3 \in \sigma_3$, we define $V_1^3$, $V_2^3$, and $V_3^3$ to be the triangle with open edge conv$\{0, e_1, e_2\} \setminus$ conv$\{0, e_1\}$. Moreover, since $4 \not\in \sigma_3$, we take the set $V_4^3$ to be empty.
\item Step 4: \emph{After exhausting all of our nonempty codewords, we finish by defining $U_j$ as $$U_j = \bigcup_{i=1}^k V_j^i.$$} In Figure 3, we obtain the $U_i$ in the Step 4 row by unioning down each column. The essential part of this construction is that for all $j$ and $i \neq \ell$, we have that $V_j^i \cap V_j^{\ell} = \emptyset$. Thus, when we union over corresponding sets, we do not accumulate extra codewords that are not in our original code. Therefore, our final collection of $U_i$ realize our nonempty codewords. Convexity is preserved by adding convex sets of increasing dimension at each step while still ensuring that our final $U_i$ are connected. 
\item Step 5: To finish, we reconsider $\emptyset$. Since $\emptyset \in \mathcal{C}$, we define our stimulus space to be $X = \mathbb{R}^2$. Since $\bigcup_{j=1}^4 U_j =$ conv$\{0, e_1, e_2\} \subsetneq \mathbb{R}^2$, we have that $\emptyset$ is in our corresponding code.
\end{itemize}

Extending this construction to the general case, we can show that any binary code can be realized by convex sets.

\begin{thm}
Every code is convex realizable. Moreover, for a code with $k$ nonempty codewords the minimal convex embedding dimension $d\leq k-1$. 

\end{thm}
\begin{proof} 
Let $\mathcal{C}$ be an arbitrary code on $n$ neurons and write $\mathcal{C} \setminus\{\emptyset\} = \{ \sigma_1 , \sigma_2 , \ldots \sigma_k \}$. Let $\{e_1,...,e_{k-1}\}$ be the standard basis for $\mathbb{R}^{k-1}$. 
Define the following sets:

\begin{align*} S_i &= \text{conv}\{ 0, e_1, e_2, \ldots e_{i-1}\} \setminus \text{conv}\{0, e_1, e_2, \ldots e_{i-2}\}\\ & \\
V_j^i &= \begin{cases} 
      S_i & \text{if } j\in \sigma_i \\
      \emptyset & \text{if } j\not\in \sigma_i 
   \end{cases}
   \qquad \text{and}
   \qquad
   U_j = \bigcup_{i=1}^k V_j^i.
\end{align*}

We claim that $\mathcal{U} = \{ U_1, U_2 , \ldots , U_n\}$ is a convex realization of $\mathcal{C}$ in a certain stimulus space $X$.

First, we will prove that each $U_j$ is convex. Let $\sigma_j \in \mathcal{C}$ be an arbitrary nonempty codeword. We will show that $U_j$, defined above, is convex by proving that $\bigcup_{i=1}^m V_j^i$ is convex for all $m$. We proceed by induction on $m$. 

Base Case: By our construction, $V_j^1$ is either empty or equal to the point at the origin, implying that $V_j^1$ is convex.

Inductive Step: Assume that $\bigcup_{i=1}^{\ell-1} V_j^i$ is convex. Then, we have that 
$$\bigcup_{i=1}^{\ell} V_j^i = V_j^{\ell}\cup\Bigg( \bigcup_{i=1}^{\ell-1}V_j^i\Bigg).$$
If $V_j^{\ell} = \emptyset$, then we are done by our inductive hypothesis. 
If $V_j^{\ell}$ is nonempty, then by construction
$V_{j}^{\ell} =$ conv$\{ 0, e_1, e_2,  \ldots, e_{\ell -1}\} \setminus$conv$\{0, e_1, e_2 , \ldots e_{\ell-2}\}$. Then, $V_j^{\ell}$ is a convex set contained in the $(\ell-1)$ dimensional simplex $A =$conv$\{ 0, e_1, e_2, \ldots , e_{\ell-1}\}$. Moreover, the set $\bigcup_{i=1}^{\ell-1} V_j^i$ is entirely contained in a facet of $A$, specifically the facet conv$\{0, e_1, e_2, \ldots e_{l-2} \}$, along which $V_j^{\ell}$ is open. This implies that $\bigcup_{i=1}^{\ell} V_j^i$ is convex if and only if $\bigcup_{i=1}^{\ell-1} V_j^i$ is convex. By our inductive hypothesis, $\bigcup_{i=1}^{\ell-1} V_j^i$ is convex, and thus our result follows. 

Next, we will define the stimulus space $X$. If $\emptyset \in \mathcal{C}$, then we define $X = \mathbb{R}^{k-1}$. If $\emptyset \not\in \mathcal{C}$, then we define $X = \bigcup_{j=1}^n U_j$.

Finally, we will show that $\mathcal{C}(\mathcal{U}) = \mathcal{C}$. By construction, $$\bigcup_{j=1}^n U_j = \text{ conv}\{0, e_1, e_2, \ldots, e_{k-1}\} \subsetneq \mathbb{R}^{k-1}.$$ Thus, if $\emptyset\in \mathcal{C}$, then $\emptyset \in \mathcal{C}(\mathcal{U})$. If $\emptyset \not\in \mathcal{C}$, then $X = \bigcup_{j=1}^n U_j$, so it holds that $\emptyset\not\in \mathcal{C}(\mathcal{U})$. 
Next, from our construction,
$$ U_j = \bigcup_{i=1}^k V_j^i = \bigcup_{i \text{ s.t. } j\in \sigma_i} S_i.$$
For any nonempty $\sigma \subseteq [n]$, 
$$U_{\sigma} = \bigcap_{j\in \sigma} U_{j} = \bigcap_{j\in \sigma} \Big( \bigcup_{i \text{ s.t. } j\in \sigma_i }S_i   \Big) = \bigcup_{i \text{ s.t. } \sigma \subseteq \sigma_i} S_i,$$
which follows from the fact that the $S_i$'s are disjoint. 
Therefore, $ U_{\sigma} \neq \emptyset$ if and only if $\sigma \subseteq \sigma_i$ for some codeword $\sigma_i \in \mathcal{C}$. Similarly,
$$\bigcup_{l\not\in \sigma} U_l = \bigcup_{l\not\in\sigma} \Big(\bigcup_{i \text{ s.t. } l \in \sigma_i} S_i\Big) = \bigcup_{i \text{ s.t. } \sigma \not\subseteq \sigma_i} S_i.$$
Therefore, we conclude that 
$$U_{\sigma} \setminus \bigcup_{l\not\in \sigma} U_l  = \emptyset \hspace{3mm} \text{  if and only if  } \hspace{3mm}  \{i \mid \sigma \subseteq \sigma_i\} \cap \{i'\mid \sigma \not \subseteq \sigma_{i'}\} = \emptyset.$$
The last inclusion is equivalent to $\sigma \not \in \mathcal{C}$. 

Thus, a codeword $\sigma$ is realized by $\mathcal{U}$ with respect to $X$ if and only if $\sigma \in \mathcal{C}$. 
\end{proof}

Our next result, Theorem 2, will show that, for some codes, the dimension of the construction in Theorem 1 is the minimal convex embedding dimension. Thus, the construction in Theorem 1 is sharp in terms of dimension. Before this result, recall Helly's Theorem, the tool which we will use to prove Theorem 2.

\begin{prop}[Helly's Theorem]
Let $X_1, \ldots , X_n$ be a finite collection of convex subsets of $\mathbb{R}^k$, with $n>k$. If the intersection of every $k+1$ of these sets is nonempty, then
$$\bigcap_{j=1}^n X_j \neq \emptyset.$$ 
\end{prop}

We begin with an example that illustrates the main ideas of the upcoming proof:
\begin{exmp}
Let $\mathcal{C}_4=\{123, 124, 134, 234\}$, and let $\mathcal{U}=\{U_1, U_2, U_3, U_4 \}$ be a convex realization of $\mathcal{C}_4$. Then, there exist the following distinct points: $a_{1} \in U_{234}$, $a_{2} \in U_{134},$ $a_{3} \in U_{124}$, and $a_{4} \in U_{123}$. That is, the intersection of any group of three of the sets in $\mathcal{U}$ is nonempty. Suppose for contradiction that $\mathcal{C}_4$ has a convex realization in $\mathbb{R}^2$. By Helly's Theorem, there exists a point $x\in \bigcap_{i=1}^4 U_i$. To clarify why such a point exists and to give an idea of the proof of Theorem 2, consider Figure 4.

First, by the convexity of $U_1$, the triangle conv$\{a_{2}, a_3, a_4\}$ is contained in $U_{1}$, since $a_{2}\in U_{134} \subseteq U_{1}$, $a_{3} \in U_{124}\subseteq U_{1}$, and $a_4 \in U_{123} \subseteq U_1$. This is pictured in the top left box of Figure 4. Similarly, conv$\{a_1, a_3, a_4\} \subseteq U_2$, conv$\{a_1, a_2, a_4\} \subseteq U_3$, and conv$\{a_1, a_2, a_3\} \subseteq U_4$, as depicted in the top right, lower left, and lower right boxes of Figure 4 respectively.

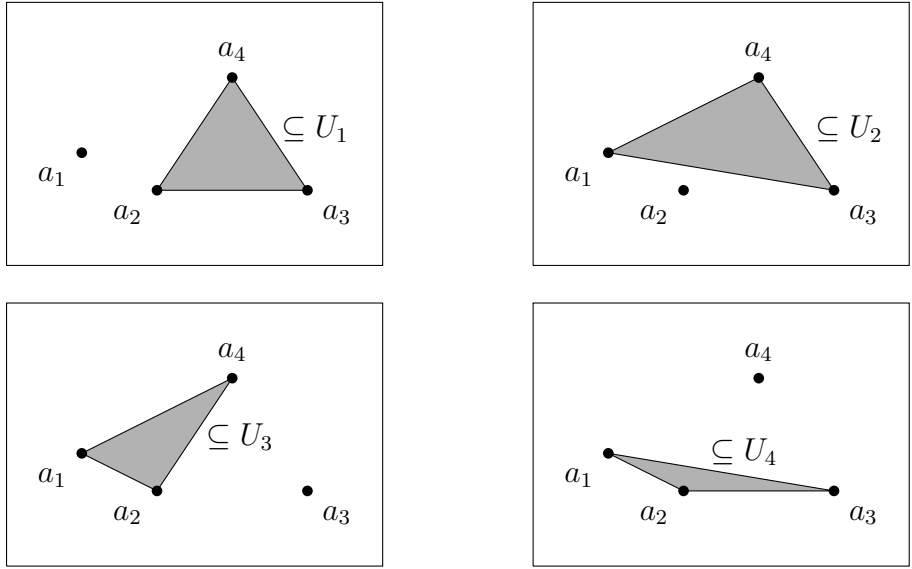
\begin{figure}[H]
\begin{center}
\begin{tikzpicture}
\fill[black!30] (0,1.5)--(1,0)--(-1,0)--(0,1.5);
\draw node[label=above:$\subseteq U_1$]at (1.1,.3) [shape=circle]{};
\draw node[fill=black, inner sep=.05cm][label=above:$a_{4}$]at (0,1.5) [shape=circle]{};
\draw node[fill=black, inner sep=.05cm][label=below right:$a_{3}$] at (1,0) [shape=circle]{};
\draw node[fill=black, inner sep=.05cm][label=below left:$a_{2}$] at (-1,0) [shape=circle]{};
\draw node[fill=black, inner sep=.05cm][label=below left:$a_{1}$]  at (-2,.5) [shape=circle]{};
\draw [-] (0,1.5)--(1,0)--(-1,0)--(0,1.5);
\draw[-] (-3, -1)--(2,-1)--(2,2.5)--(-3,2.5)--(-3,-1);

\fill[black!30][-]  (7,1.5)--(8,0)--(5,.5)--(7,1.5);
\draw node[label=above:$\subseteq U_2$]at (8.2,.3) [shape=circle]{};
\draw node[fill=black, inner sep=.05cm][label=above:$a_{4}$]at (7,1.5) [shape=circle]{};
\draw node[fill=black, inner sep=.05cm][label=below right:$a_{3}$] at (8,0) [shape=circle]{};
\draw node[fill=black, inner sep=.05cm][label=below left:$a_{2}$] at (6,0) [shape=circle]{};
\draw node[fill=black, inner sep=.05cm][label=below left:$a_{1}$]  at (5,.5) [shape=circle]{};
\draw [-]  (7,1.5)--(8,0)--(5,.5)--(7,1.5);
\draw[-] (4, -1)--(9,-1)--(9,2.5)--(4,2.5)--(4,-1);

\fill[black!30]   (8,-4)--(5,-3.5)--(6,-4)--(8,-4);
\draw node[label=above:$\subseteq U_4$]at (6.8,-4) [shape=circle]{};
\draw node[fill=black, inner sep=.05cm][label=above:$a_{4}$]at (7,-2.5) [shape=circle]{};
\draw node[fill=black, inner sep=.05cm][label=below right:$a_{3}$] at (8,-4) [shape=circle]{};
\draw node[fill=black, inner sep=.05cm][label=below left:$a_{2}$] at (6,-4) [shape=circle]{};
\draw node[fill=black, inner sep=.05cm][label=below left:$a_{1}$]  at (5,-3.5) [shape=circle]{};
\draw [-]  (8,-4)--(5,-3.5)--(6,-4)--(8,-4);
\draw[-] (4, -5)--(9,-5)--(9,-1.5)--(4,-1.5)--(4,-5);

\fill[black!30]  (0,-2.5)--(-2,-3.5)--(-1,-4)--(0,-2.5);
\draw node[label=above:$\subseteq U_3$]at (.1,-3.8) [shape=circle]{};
\draw node[fill=black, inner sep=.05cm][label=above:$a_{4}$]at (0,-2.5) [shape=circle]{};
\draw node[fill=black, inner sep=.05cm][label=below right:$a_{3}$] at (1,-4) [shape=circle]{};
\draw node[fill=black, inner sep=.05cm][label=below left:$a_{2}$] at (-1,-4) [shape=circle]{};
\draw node[fill=black, inner sep=.05cm][label=below left:$a_{1}$]  at (-2,-3.5) [shape=circle]{};
\draw [-] (0,-2.5)--(-2,-3.5)--(-1,-4)--(0,-2.5);
\draw[-] (-3, -5)--(2,-5)--(2,-1.5)--(-3,-1.5)--(-3,-5);

\end{tikzpicture}
\caption{Finding a point $x\in \bigcap_{i=1}^4 U_i$.}
\end{center}
\end{figure}

Superimposing the above four boxes, we obtain the following:

\begin{figure}[H]
\begin{center}
\begin{tikzpicture}
\fill[black!30] (0,1.5)--(-2,0.5)--(-1,0)--(1,0)--(0,1.5);
\draw node[fill=black, inner sep=.05cm][label=above:$x$]at (-.77,.32) [shape=circle]{};
\draw node[fill=black, inner sep=.05cm][label=above:$a_{4}$]at (0,1.5) [shape=circle]{};
\draw node[fill=black, inner sep=.05cm][label=below right:$a_{3}$] at (1,0) [shape=circle]{};
\draw node[fill=black, inner sep=.05cm][label=below left:$a_{2}$] at (-1,0) [shape=circle]{};
\draw node[fill=black, inner sep=.05cm][label=below left:$a_{1}$]  at (-2,.5) [shape=circle]{};
\draw [-] (0,1.5)--(1,0)--(-1,0)--(0,1.5)--(-2,0.5)--(1,0);
\draw[-] (-2,.5)--(-1,0);
\draw[-] (-3, -1)--(2,-1)--(2,2.5)--(-3,2.5)--(-3,-1);
\end{tikzpicture}
\caption{Superimposing the four boxes in Figure 4.}
\end{center}
\end{figure}
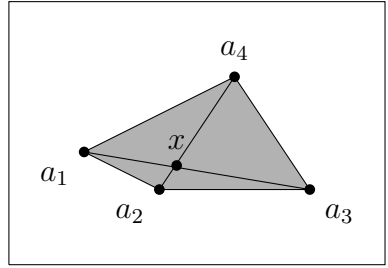

Thus, $x\in \bigcap_{i=1}^4 U_i$, so the codeword $1234$ is realized at $x$, which is a contradiction since $1234 \not\in \mathcal{C}_4$. Although Figures 4 and 5 only provide one case of the placement of four points, Helly's Theorem holds regardless of where these points are placed in $\mathbb{R}^2$, so such an $x$ will always exist. Thus, by Helly's Theorem, the minimal convex embedding dimension of $\mathcal{C}_4$ must be greater than 2 and thus equals 3 by Theorem 1.
\end{exmp}

\begin{defn}
Let $\mathcal{C}_n$ be the code on $n$ neurons containing exactly all codewords of length $n-1$, \begin{equation}\mathcal{C}_n := \{ \sigma\subseteq [n] \mid |\sigma| = n-1\}.\tag{$\ast$}
\end{equation}
\end{defn}

The code $\mathcal{C}_n$ contains exactly the codewords of length $n-1$, so $|\mathcal{C}_n| =$ ${n} \choose {n-1}$ $=n$. In Theorem 2, we prove that $\mathcal{C}_n$ has minimal embedding dimension $n-1$.

\begin{thm}
Let $\mathcal{C}_n$ be the code on $n$ neurons as defined in $(\ast)$. Then, the minimal convex embedding dimension of $\mathcal{C}_n$ is $n-1$. That is, the embedding dimension from Theorem 1 is exactly the minimal convex embedding dimension of $\mathcal{C}_n$ for every $n$. 
\end{thm}
\begin{proof}

Let $n$ be arbitrary. Let $\{U_1 , U_2 ,\ldots , U_n\}$ be a convex realization of $\mathcal{C}_n$.

Assume for contradiction that $U_1, U_2, \ldots, U_n$ all lie in a copy of $\mathbb{R}^k$ where $k\leq n-2$. Every collection of $(n-1)$ of the $U_i$'s intersect by construction of $\mathcal{C}_n$, so by Helly's Theorem, there exists a point $x\in \bigcap_{i=1}^n U_i$. Thus, the codeword $[n]$ is realized at $x$, which is a contradiction as $[n] \not\in \mathcal{C}$.
Therefore, $\mathcal{C}_n$ is not convex realizable in $\mathbb{R}^{k}$ for $k\leq n-2$, implying by Theorem 1 that its minimal embedding dimension is $n-1$. 
\end{proof}

\begin{cor}
The minimal convex embedding dimension of all neural codes has no upper bound.
\end{cor}
\begin{proof}
This follows immediately from Theorem 2.
\end{proof}

Lastly, we make a conjecture pertaining to codes with minimal convex embedding dimension 1. This conjecture is related to a result of Rosen and Zhang. Rosen and Zhang characterized codes that are realizable by open convex sets in dimension 1 \cite{d-1}. We speculate that these codes are precisely the codes with minimal convex embedding dimension 1 (Conjecture 1). Conjecture 1 is equivalent to Conjecture 2, which provides a sufficient condition for an open convex code to have a minimal convex embedding dimension of 2.

\begin{defn}
Let $\mathcal{C}$ be an open convex code. The smallest dimension $k$ in which $\mathcal{C}$ is realizable by open convex sets in $\mathbb{R}^k$ is its \textit{minimal open convex embedding dimension}. 
\end{defn}

\begin{conjecture}
A code $\mathcal{C}$ has minimal convex embedding dimension 1 if and only if it has minimal \underline{open} convex embedding dimension 1. 
\end{conjecture}

We propose a procedure to obtain an open convex realization in dimension 1 from a code which has minimal convex embedding dimension of 1, as follows:

Let $\mathcal{C}$ be a binary code on $n$ neurons with minimal convex embedding dimension $d=1$, and let $\mathcal{U}=\{I_1,...,I_n\}$ be a convex realization of $\mathcal{C}$ in $\mathbb{R}^1$. We denote the left and right endpoints of the interval $I_k$ by  $a_k$ and $b_k$, respectively. 
Define $\varepsilon$ to be the smallest non-zero distance between any two endpoints. Next we will modify the endpoints of each $I_k\in \mathcal{U}$ to obtain a new open interval $I_k'=(a_k',b_k')$, where 
\begin{equation*}
a_k':=\begin{cases} 
a_k-\varepsilon/3 & \text{if } a_k \in{I_k}\\
a_k+\varepsilon/3 & \text{if } a_k \not\in{I_k}
\end{cases} \quad \text{ and } \quad 
b_k':=\begin{cases} 
b_k+\varepsilon/3 & \text{if } b_k \in{I_k}\\
b_k-\varepsilon/3 & \text{if } b_k \not\in{I_k}
\end{cases}.
\end{equation*}

In other words, we are shrinking $I_k$ at open endpoints and extending $I_k$ at closed endpoints. Let $\mathcal{U}' := \{I_k' = (a_k', b_k') \mid I_k \in \mathcal{U}\}$. We conjecture that $\mathcal{C}(\mathcal{U}) = \mathcal{C}(\mathcal{U'})$. 

However, in order to verify that $\mathcal{C}(\mathcal{U}) = \mathcal{C}(\mathcal{U'})$, many cases must be checked. For example, to show the inclusion $\mathcal{C}(\mathcal{U}) \subseteq \mathcal{C}(\mathcal{U'})$, it is necessary to consider codewords realized in $\mathcal{C}(\mathcal{U})$ at a single point or over an interval of nonzero length separately.

Conjecture 1 is equivalent to the following conjecture:

\begin{conjecture}
Suppose $\mathcal{C}$ is convex open and has a minimal open convex embedding dimension of 2. Then the minimal convex embedding dimension of $\mathcal{C}$ is 2.
\end{conjecture}

\section{Discussion}
We have proven by construction that every code $\mathcal{C}$ has a convex realization in $\mathbb{R}^{k-1}$ where $k$ is the number of nonempty codewords in $\mathcal{C}$ (Theorem 1). This result shows the necessity of assuming in applications openness or closedness of a convex realization of a neural code. That is, one cannot assume that receptive fields are general convex sets, or even general connected sets, as otherwise every neural code would be realizable.

Moving forward, since we have shown that every binary code is convex realizable, can we determine the minimal convex embedding dimension? Is every code convex realizable by sets of codimension 0? If a code is convex open or closed, when is the minimal open or closed convex embedding dimension strictly greater than the minimal embedding dimension? 

Theorem 2 provides some framework for answering these questions. Theorem 2 implies that, in certain cases, the dimension of the construction in Theorem 1 is exactly the minimal convex embedding dimension. Moreover, Theorem 2 implies that there is no upper bound on the minimal convex embedding dimensions of all codes. 

Theorem 2 inspires another interesting question, can $\mathcal{C}_n$ be realized by full-dimensional convex sets (receptive fields are typically full-dimensional)? That is, can $\mathcal{C}_n$ be realized in $\mathbb{R}^{k-1}$ by convex sets $U_1, \ldots, U_{n}$ where each $U_i$ contains an $(k-1)$-dimensional ball? We expect that such a realization does not exist. 
\section*{Acknowledgements}
This research was supported by the NSF-funded REU program at Texas A\&M University [NSF DMS-1460766]. We would like to thank our research mentor Dr. Anne Shiu for her guidance and indispensable insight. We would also like to thank Ola Sobieska for her assistance, as well as Dr. Carina Curto, Dr. Jack Jeffries, and Benjamin Spitz for providing feedback on earlier drafts. We also thank an anonymous for detailed suggestions that improved this work. Lastly, we would like to thank Dr. Vladimir Itskov for helpful discussions, including ideas that contributed to the proof that $\mathcal{C} = \mathcal{C}(\mathcal{U})$ in Theorem 1.

\bibliography{demo}
\bibliographystyle{abbrv}

\end{document}